\documentclass[a4paper, 11pt]{article}
\usepackage{amsmath, amsthm, amssymb, fullpage, graphicx, enumerate}

\newtheorem{thm}{Theorem}
\newtheorem{lem}[thm]{Lemma}
\newtheorem{cor}[thm]{Corollary}
\newtheorem{problem}[thm]{Problem}
\theoremstyle{remark}
\newtheorem*{rem}{Remark}

\newcommand*{\abs}[1]{\lvert#1\rvert}
\newcommand*{\floor}[1]{\lfloor#1\rfloor}
\newcommand*{\pc}[1][c]{\dot p_{\mathrm #1}}
\newcommand{\pairs}{\mathfrak O}
\newcommand{\vecom}{\boldsymbol\omega}
\newcommand*{\prob}[2][p]{\mathbb P_{#1}(#2)}
\newcommand*{\mean}[2][p]{\mathbb E_{#1}(#2)}
\newcommand*{\bigbr}[1]{\bigl(#1\bigr)}
\newcommand*{\biggerbr}[1]{\biggl(#1\biggr)}
\newcommand*{\bigmean}[2][p]{\mathbb E_{#1}\bigbr{#2}}
\newcommand*{\theorem}[1]{Theorem~\ref{#1}}
\newcommand*{\lemma}[1]{Lemma~\ref{#1}}
\newcommand*{\corollary}[1]{Corollary~\ref{#1}}

\newcommand{\ie}{i.e.\ }
\newcommand{\eg}{e.g.\ }
\newcommand*{\bound}[1][n]{\floor{#1^2/12-#1/2+1}}

\title{Site percolation and isoperimetric inequalities for plane graphs}
\usepackage{authblk}
\author{John Haslegrave}
\author{Christoforos Panagiotis}
\affil{{Mathematics Institute}\\
	{University of Warwick}\\
	{CV4 7AL, UK}}

\begin{document}
\maketitle

\begin{abstract}We use isoperimetric inequalities combined with a new technique to prove upper bounds for the site percolation threshold of plane graphs
with given minimum degree conditions. In the process we prove tight new isoperimetric bounds for certain classes of hyperbolic graphs. 
This establishes the vertex isoperimetric constant for all triangular and square hyperbolic lattices, answering a question of Lyons and Peres. 

We prove that plane graphs of minimum degree at least $7$ have site percolation threshold bounded away from $1/2$, which was conjectured by Benjamini and Schramm, 
and make progress on a conjecture of Angel, Benjamini and Horesh that the critical probability
is at most $1/2$ for plane triangulations of minimum degree $6$. We prove additional bounds for stronger minimum degree conditions, and for graphs without triangular faces.
\end{abstract}

\section{Introduction}
In their highly influential paper \cite{beyond}, Benjamini and Schramm made several conjectures that generated a lot of interest among mathematicians and led to many beautiful mathematical results \cite{cutsets,NoPerco,PercoHyperbolic,PhaseTransitionGroups,puMonotonicity,
NoPercoExp}, just to name a few. Despite the substantial amount of work, most of these conjectures are still open, while for a few of them, hardly anything is known. One of their conjectures states that $\pc(G)<1/2$ on any planar graph $G$ of minimal degree at least $7$; they additionally conjecture that there are infinitely many infinite open clusters on the interval $(\pc(G),1-\pc(G))$. As Benjamini and Schramm observe in their paper, every planar graph of minimal degree at least $7$ is non-amenable. The conjecture has been verified for the $d$-regular triangulations of the hyperbolic plane in \cite{PercoHyperbolic}.

The connection between percolation thresholds and isoperimetric (or Cheeger) constants is well known, and in \cite{beyond} it is proved that the site percolation threshold
for a graph $G$ is bounded above by $(1+\dot{h}(G))^{-1}$, where $\dot{h}(G)$ is the vertex isoperimetric constant. In their book \cite{LyonsBook}, Lyons and Peres
give the edge isoperimetric constants for the regular hyperbolic tessellations $H_{d,d'}$, where $(d-2)(d'-2)>4$ (which were established by H\"{a}ggstr\"{o}m, Jonasson and Lyons \cite{HJL}), and ask \cite[Question~6.20]{LyonsBook} for the corresponding vertex isoperimetric constants.

Angel, Benjamini and Horesh considered isoperimetric inequalities for plane triangulations of minimum degree $6$ in \cite{ABH}, and proved a discrete analogue 
of Weil's theorem, showing that any such triangulation satisfies the same isoperimetric inequality as the Euclidean triangular lattice $\mathbb T_6$. 
They conjectured that $\mathbb T_6$ is extremal in other ways which might be expected to have connections with isoperimetric properties. 
First, they conjecture that the connective constant $\mu(T)$ -- that is, the exponential growth rate of the number of self-avoiding walks 
of length $n$ on $T$ -- is minimised by $\mathbb T_6$ among triangulations of minimum degree at least $6$. Secondly, they conjecture that percolation is hardest to achieve 
on $\mathbb T_6$ in the sense that both the critical probability for bond percolation $p_{\mathrm c}(T)$ and the critical probability 
for site percolation $\pc(T)$ are maximised by $\mathbb T_6$. Intuitively these conjectures are connected, in that if fewer 
long self-avoiding paths exist then long connections might be expected to be less robust, making percolation less likely to occur. See also \cite{Uniformization} for several other conjectures regarding planar triangulations. 

Georgakopoulos and Panagiotis \cite{analyticity} proved that $p_{\mathrm c}(T)\leq 1/2$ for any planar triangulation $T$ with minimum degree at least $6$, and a well-known result of Grimmett and Stacey \cite{Grimmett} shows that $\pc(T)\leq 1-(1/2)^d$ when the degrees in $T$ are bounded by $d$. The latter bound can be improved to $\pc(T)\leq 1-\frac{1}{d-1}$ \cite{analyticity} but in both cases, the bounds converge to $1$ as the maximal degree converges to infinity. We remark that Benjamini and Schramm \cite{beyond} made an even stronger conjecture than that one of Angel, Benjamini and Horesh mentioned above, namely that $\pc(T)\leq 1/2$ for any planar triangulation without logarithmic cut sets. 

In section \ref{triang} we consider the conjecture of Angel, Benjamini and Horesh for site percolation, and we prove the following theorem.
\begin{thm}\label{first}For any plane graph $G$ of minimum degree at least $6$,\[\pc(G)\leq2/3\,.\]\end{thm}
\noindent
In section \ref{hyper} we study plane graphs of minimal degree at least $7$ and we prove the aforementioned conjecture of Benjamini and Schramm.
\begin{thm}Let $G$ be a plane graph with minimum degree at least $d\geq7$. Then \[\pc(G)\leq\frac{2+\alpha_d}{(d-3)(1+\alpha_d)}\] where $\alpha_d=\frac{d-6-\sqrt{(d-2)(d-6)}}{2}$. In particular, for plane graphs with minimum degree at least $7$ \[\pc(G)\leq \frac{2+\alpha_7}{4(1+\alpha_7)}\approx 0.3455 \]\end{thm}
For plane graphs without faces of degree $3$, a minimum vertex degree of $5$ is sufficient to ensure non-amenability; in section \ref{quad} we give bounds on the site percolation threshold in this case.
\begin{thm}Let $G$ be a plane graph with minimum degree at least $d\geq5$ and minimum face degree at least $4$. Then \[\pc(G)\leq\frac{(2+\alpha_{d+2})(d-2)}{(1+\alpha_{d+2})(d^2-3d+1)}.\]\end{thm}
\noindent
As far as we know, these results are new even for the $d$-regular triangulations and quadrangulations of the hyperbolic plane. In the process, we obtain best-possible bounds on the vertex Cheeger constants $\dot h(G)$ of such graphs.
\begin{thm}
Let $G$ be a plane graph with minimum degree at least $d\geq7$. Then
\[\dot h(G)\geq \alpha_d.\] If $G$ is the $d$-regular triangulation of the hyperbolic plane, then we have equality.
\end{thm}
\begin{thm}
Let $G$ be a plane graph with minimum degree at least $d\geq5$ and minimum face degree at least $4$. Then
\[\dot h(G)\geq \alpha_{d+2}.\] If $G$ is the $d$-regular quadrangulation of the hyperbolic plane, then we have equality.
\end{thm}
\noindent
The second halves of the last two theorems answer the aforementioned question of Lyons and Peres for all cases with $d'=3$ and $d'=4$. In particular, the vertex isoperimetric constant for both the $7$-regular triangulation and the $5$-regular quadrangulation, $\alpha_7$, is the golden ratio.

Our results on percolation thresholds follow from bounding the `surface to volume ratio' of \textit{outer interfaces}. This notion was introduced in \cite{analyticity} to prove that in $2$-dimensional Bernoulli percolation, the percolation density is an analytic function on the interval $(p_c,1]$. In the case of triangulations, outer interfaces coincide with the standard notion of inner vertex boundary appearing in the literature. See also \cite{ExpGrowth} for some further interesting properties and open problems regarding them.

\section{Definitions and main technique}

Let $G$ be an infinite connected locally finite plane graph, and fix a root vertex $o$. Throughout, we assume that $G$ is embedded in the plane without accumulation points, \ie only finitely many vertices lie in the region enclosed by any cycle of the graph. Site percolation with intensity $p$ on $G$ is a random function $\vecom:V(G)\mapsto\{0,1\}$ for which $\{\omega(v):v\in V(G)\}$ 
are independent Bernoulli variables with parameter $p$. We say that a vertex $v$ is \textit{occupied} in an instance $\vecom$
if $\omega(v)=1$ and \textit{unoccupied} otherwise. If $o$ is occupied, the \textit{occupied cluster} of $o$ is the component of 
$o$ in the subgraph induced by all occupied vertices. The critical probability $\pc(G)$ is the infimum of
all intensities for which there is a positive probability of this cluster being infinite, which does not depend on $o$.

Let $C_o$ be any finite connected induced subgraph containing $o$. The \textit{outer interface} of $C_o$ consists of all vertices of $C_o$ 
meeting the unbounded face of $C_o$. Deleting all vertices of $C_o$ from $G$ divides the remaining graph into components, 
at least one of which lies in the unbounded face of $C_o$. Write $C_{\infty}$ for the union of all components lying in the unbounded face of $C_o$.
The \textit{outer boundary} of $C_o$ is the set of vertices in $C_\infty$ adjacent to $C_o$. 
Denote the outer interface by $M$ and the outer boundary by $B$. By definition, $M$ induces a connected subgraph of $G$ and $B$ forms a vertex
cut separating $o$ from infinity. See Figure \ref{grids} for an example; in this case $C_\infty$ is not connected but there is also another component not in $C_\infty$.

\begin{figure}
\begin{center}
\includegraphics[width=.7\textwidth]{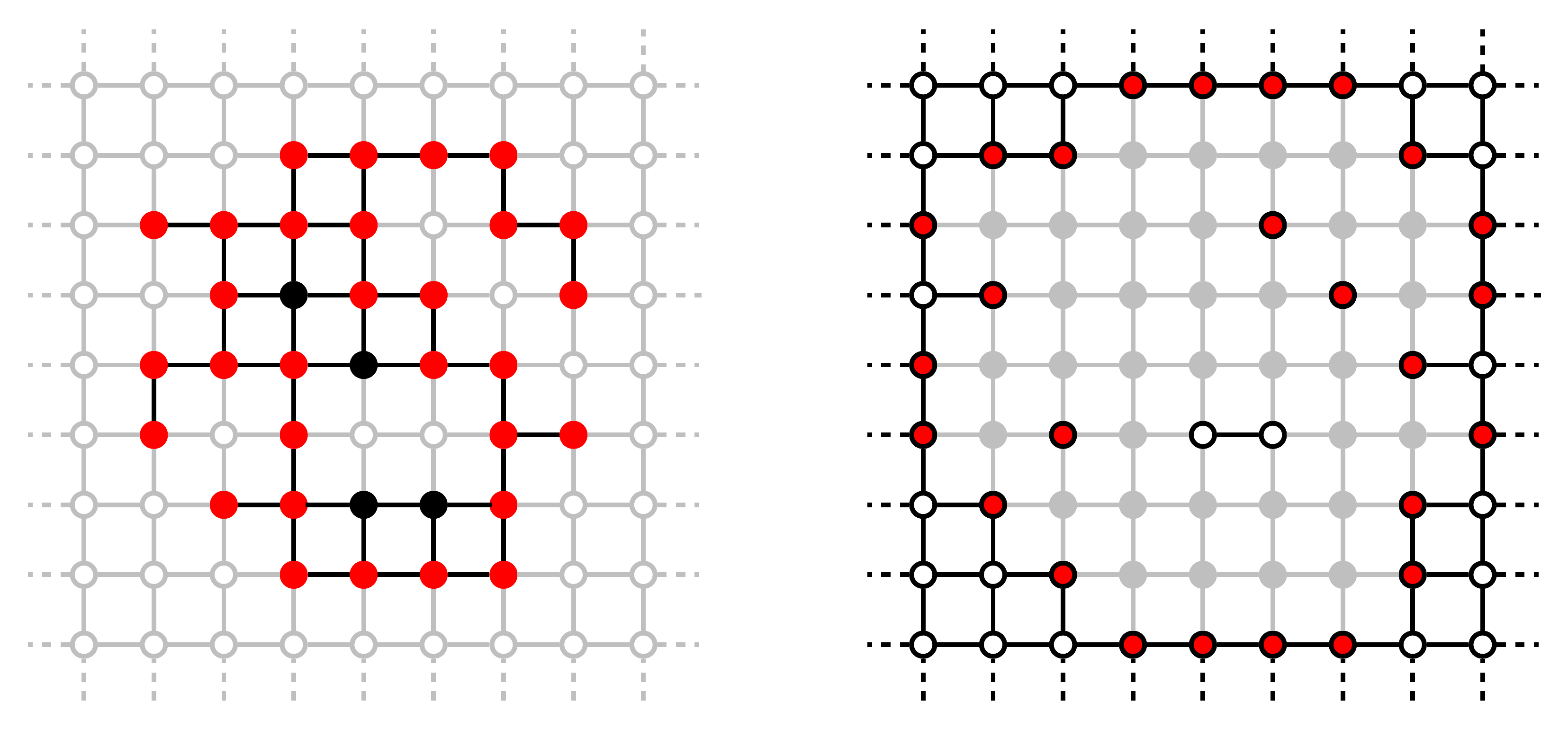}
\end{center}
\caption{A connected induced subgraph $C_o$, with its outer interface highlighted (left) and the remaining graph with the corresponding outer boundary highlighted (right).}\label{grids}
\end{figure}

Note that, while (in general) neither $M$ nor $B$ uniquely determine $C_o$, each uniquely determines $C_\infty$. 
In fact, it is the union of components of $G\setminus M$ lying in the outer face of $M$, and it is also the union of $B$ and the set of vertices 
not connected to $o$ in $G\setminus B$. Since $M$ is also the outer interface of $G\setminus C_\infty$ and $B$ coincides with the set of vertices in $C_\infty$ adjacent to $M$, 
each of $M$ and $B$ uniquely determines the other. Let the set of feasible pairs $(M,B)$ be $\pairs$, 
and for each $n$ let $\pairs_n=\{(M,B)\in\pairs:\abs B=n\}$. 

We say that a pair $(M,B)\in\pairs$ \textit{occurs} in a site percolation instance $\vecom$ if $\omega(m)=1$ for each 
$m\in M$ and $\omega(b)=0$ for each $b\in B$. Note that, in a site percolation instance $\vecom$ with $\omega(o)=1$, 
the occupied cluster of $o$ is infinite if and only if no pair $(M,B)\in\pairs$ occurs. Since each outer boundary forms a vertex cut
separating $o$ from infinity, the occurrence of a pair certainly precludes $o$ being in an infinite cluster, whereas if $o$ is
in a finite cluster $C_o$ then the outer interface and outer boundary of $C_o$ form an occurring pair.

Our main technique is to upper bound the ratio $\abs M/\abs B$ for $(M,B)\in\pairs$ and consequently to show that the probability of occurrence
of any given pair is decreasing for $p$ above a certain value. Provided $G$ satisfies an isoperimetric inequality of moderate strength,
we then deduce that with positive probability no pair occurs. This latter condition essentially requires non-positive curvature, and so
the smallest minimum degree to guarantee this is $6$ for the general case and $5$ for the triangle-free case.

\begin{thm}\label{main}Suppose that there exist a real number $\alpha$ and a function $f(n)$ of sub-exponential growth with the following properties.
\begin{enumerate}[(i)]\item\label{ratio-condition} For each pair $(M,B)\in\pairs$, we have $\abs M\leq\alpha\abs B$.
\item\label{poly-condition} For each outer boundary $B$, the component of $o$ in $G\setminus B$ contains at most $f(\abs B)$ vertices. 
\end{enumerate}
Then $\pc(G)\leq \frac\alpha{1+\alpha}$.\end{thm}

We will use a Peierls-type argument (see \eg \cite[Theorem 4.1]{Pete}) to show that with positive probability $o$ is in an infinite
component for site percolation with any intensity $p>\frac\alpha{1+\alpha}$. It is sufficient to show that with positive probability no
pair $(M,B)\in\pairs$ occurs; in fact we shall find it more convenient to work with a slightly weaker notion than occurrence.
The key observations are that not too many pairs of any given size can occur for $p=\frac\alpha{1+\alpha}$, and that the expected number of
occurring pairs at any given higher intensity is exponentially smaller.

In $G$, pick a geodesic $R$ from $o$, going to infinity. For any $(M,B)\in\pairs$, let $R_B$ be the longest initial
segment of $R$ which does not intersect $B$ (so the next vertex on $R$ is the first intersection with $B$, which must
exist since $B$ is a vertex cut), and define $M'=M\setminus R_B$. We say that the pair $(M,B)$ \textit{almost occurs} 
in a percolation instance if the vertices of $M'$ are occupied and the vertices of $B$ are unoccupied. 

We first need to show that \eqref{poly-condition} gives a bound on the number of almost-occurring pairs.
\begin{lem}\label{poly}At most $f(n)$ elements of $\pairs_n$ almost occur in any instance $\vecom$.\end{lem}
\begin{proof}Fix an instance $\vecom$. Suppose $(M,B)\in\pairs_n$ almost occurs in $\vecom$. Now define a new instance
$\vecom'$ by setting each vertex in $R_B$ to be occupied, leaving the states of other vertices unchanged. Note that $(M,B)$
occurs in $\vecom'$; in fact, since $M\cup R_B$ induces a connected subgraph, $(M,B)$ is the outer interface and boundary of 
the occupied cluster of $o$ in $\vecom'$. Thus $\vecom'$ uniquely determines $(M,B)$.

Since $o\in R_B$ and $R_B$ lies entirely within the component of $o$ in $G\setminus B$, by \eqref{poly-condition}
we have $1\leq\abs{R_B}\leq f(n)$. Thus, given $\vecom$, there are at most $f(n)$ 
possibilities for $\vecom'$ and hence at most this many pairs $(M,B)\in\pairs_n$ almost occur.\end{proof}

\begin{proof}[Proof of \theorem{main}]
We can always assume that $f$ is a strictly increasing function as otherwise we can work with $g(n):=n+\max_{k\leq n} f(k)$ which is a strictly increasing function and also satisfies the properties of $f$ in the statement of the theorem.

Let $b_{n,m}$ be the number of pairs $(M,B)\in\pairs_n$ for which $\abs{M'}=m$. By \eqref{ratio-condition},
whenever $(M,B)\in\pairs_n$ we have $\abs M\leq \alpha n$, and so $b_{n,m}=0$ for $m>\alpha n$.

Let $X_n$ be the number of pairs $(M,B)\in\pairs_n$ which almost occur, and write $q=\frac{\alpha}{1+\alpha}$. 
By \lemma{poly}, $X_n\leq f(n)$, and hence $\mean[q]{X_n}\leq f(n)$. Thus
\begin{align*}f(n)&\geq\sum_{(M,B)\in\pairs_n}\prob[q]{(M,B)\text{ almost occurs}}\\
&=\sum_{m\leq\alpha n}b_{n,m}q^m(1-q)^n\,.\end{align*}
Now for any $p>q$ we have
\begin{align*}\mean{X_n}&=\sum_{m\leq\alpha n}b_{n,m}p^m(1-p)^n\\
&=\sum_{m\leq\alpha n}b_{n,m}q^m(1-q)^n(p/q)^m\biggerbr{\frac{1-p}{1-q}}^n\\
&\leq\sum_{m\leq\alpha n}b_{n,m}q^m(1-q)^n.(p/q)^{\alpha n}\biggerbr{\frac{1-p}{1-q}}^n\\
&\leq\biggerbr{\frac{(1+\alpha)^{1+\alpha}}{\alpha^\alpha}p^\alpha(1-p)}^nf(n)\,.\end{align*}
The arithmetic-geometric mean inequality implies $p^\alpha(1-p)<\frac{\alpha^\alpha}{(1+\alpha)^{1+\alpha}}$. Consequently, since $f(n)$ is sub-exponential, 
$\sum_{n>0}\mean{X_n}$ is finite, and in particular there is some $n_0$ such that $\bigmean{\sum_{n\geq n_0}X_n}<1$. 

Let $\vecom$ be a site percolation instance with intensity $p$. Notice that any pair $(M,B)$ with $\abs{R_B}\geq f(n_0)$ satisfies $\abs{B}\geq n_0$ by the monotonicity of $f$. We can now deduce from the union bound that with positive probability no  
pair $(M,B)$ with $\abs{R_B}\geq f(n_0)$ almost occurs in $\omega$; call this event $E_1$. Let $E_2$ be the event that the first $f(n_0)$ vertices along $R$ are all occupied in $\vecom$; clearly $\prob{E_2}>0$. As $E_1$ and $E_2$ are determined by the states of disjoint sets of vertices, they are independent, and so their intersection also has positive probability. In the event $E_1\cap E_2$, no pair $(M,B)$ occurs in $\omega$. Thus $o$ is in an infinite cluster with positive probability, so $p\geq\pc(G)$.
\end{proof}

Following \cite{beyond}, we use the notation $\partial S$, where $S$ is a set of vertices, to denote the set of vertices
which are not in $S$ but are adjacent to some vertex in $S$, and we define the (site) Cheeger constant
\[\dot h(G)=\inf_{\abs S<\infty}\frac{\abs{\partial S}}{\abs S}\,.\]

\section{Graphs of minimum degree at least 6}\label{triang}

In this section we will consider the following problem of Angel, Benjamini and Horesh \cite{ABH}. 

\begin{problem}[{from \cite[Problem 4.2]{ABH}}]\label{pc-half} 
If $T$ is a plane triangulation with all vertex degrees at least $6$, is it necessarily true that $\pc(T)\leq 1/2$?
\end{problem}
\noindent
The authors ask also similar questions for bond percolation and self-avoiding walks. 

While \cite{ABH} does not give a precise definition of ``plane triangulation'', it is clear from context that accumulation points are not permitted. Without this assumption, the answer to the question would be negative. Indeed, consider the graph of an infinite cylinder formed by stacking congruent antiprisms, which can be embedded in the plane with a single accumulation point. Notice that all vertices have degree $6$, and furthermore this graph contains infinitely many disjoint cut sets of fixed size separating $o$ from infinity, hence $\pc=1$ by the Borel-Cantelli lemma. We remark in passing that the same graph demonstrates that \cite[Theorem 2.4]{ABH} also assumes accumulation points are not permitted. With this assumption, we firmly believe that the answer to Problem \ref{pc-half} is affirmative.

The aim of this section is to give upper bounds for $\pc$ on general plane graphs of minimum degree at least $6$ that are not necessarily triangulations. As we will see, the general case can be easily reduced to the case of triangulations. 

The following two results of Angel, Benjamini and Horesh \cite{ABH} about plane triangulations $T$ with finitely many vertices will be used in our proofs. A \textit{plane triangulation} is a plane graph without accumulation points in which every bounded face has degree $3$. The \textit{boundary} of $T$ is the set of vertices incident with the unbounded face. The remaining vertices of $T$ are called \textit{internal}. The \textit{total boundary length} of $T$ counts all edges induced by the boundary vertices of $T$ exactly once, except for those not incident with a triangular face of $T$ which are counted twice. When the boundary vertices of $T$ span a cycle, we will say that its boundary is \textit{simple} and $T$ is a \textit{disc triangulation}. 

\begin{lem}\label{iso-layer}Let $T$ be a disc triangulation with a simple boundary of length $n$ and at least one internal vertex. 
Let $T'$ be the triangulation induced by the internal vertices of $T$ and let $m$ be the total boundary length of $T'$.
Suppose all internal vertices of $T$ have degree at least $6$. Then $m\leq n-6$.\end{lem}
\begin{lem}\label{iso-volume}Any disc triangulation with $k$ vertices and $n$ boundary vertices, and with all internal vertices
having degree at least $6$, satisfies $k\leq\floor{n^2/12+n/2+1}$.
\end{lem}

For the remainder of this section, we fix some plane graph $G$ of minimum degree at least $6$. For any pair $(M,B)\in\pairs_n$, consider the subgraph of $G$ induced by $B$ and the finite components of $G\setminus B$. Write $B^\circ$ for the set of vertices in $B$ which were adjacent to the infinite component of $G\setminus B$; note that $\abs{B^\circ}\leq\abs{B}=n$. By adding edges, if necessary, to the subgraph we may obtain a plane triangulation $T=T(M,B)$ with finitely many vertices, say $k$, and boundary $B^\circ$. We do this by first adding edges joining vertices of $B^\circ$ cyclically so that no other vertices meet the unbounded face, then adding internal edges to triangulate bounded faces of degree greater than $3$. 
Each time that we triangulate a face between $B$ and $M$, we do so by adding all diagonals meeting some vertex of $B$; this will ensure that in the final triangulation the subgraph induced by $B$ is connected, and every vertex of $M$ is adjacent to $B$.
Let us fix such a triangulation $T$. Since $M$ is connected, $T$ must be a disc triangulation, and the choice of $B^\circ$ ensures that all internal vertices have degree at least $6$. The next corollary follows now from \lemma{iso-volume}, noting that the number of internal vertices is at most $k-n$.
\begin{cor}\label{volume}For each pair $(M,B)\in\pairs_n$, the component of $o$ in $G\setminus B$ consists of at most $\bound$ vertices.\end{cor}
We wish to apply \lemma{iso-layer} to bound $\abs M$ in terms of $\abs B$, for $(M,B)\in\pairs$. However, the boundary of $T(M,B)$ does not necessarily coincide with $B$; to deal with this we show that the triangulation can be modified to give a disc triangulation with boundary which is not too much larger than $\abs{B}$. First we need a simple application of Euler's formula.
\begin{lem}\label{boundary}Consider a pair $(M,B)\in\pairs$ and the corresponding triangulation $T=T(M,B)$. Let $H$ be the subgraph of $T$ induced by $B$, and $f$ a face of $H$. Write $\partial f$ for the boundary of $f$,
and write $\partial^\circ f\subseteq\partial f$ for that part of the boundary which forms a cycle separating $f$ from infinity. Follow $\partial f$ clockwise, writing down a list of vertices visited (so the same vertex can appear in the list multiple times). The length of the list so obtained is at most $2\abs{\partial f}-\abs{\partial^\circ f}$, which is at most $2\abs{\partial f}-3$.\end{lem}
\begin{proof}
We may assume that $f$ is an internal face. Add a new vertex inside $f$, and join it to each vertex in the list in turn. This gives a plane multigraph $H'$ with $\abs{\partial f}+1$ vertices in which each face incident with 
the new vertex has degree $3$ and, since $H$ was simple, all other faces have degree at least $3$. The unbounded face has degree 
$\abs{\partial^\circ f}$, but there may be other faces inherited from $H$. 

Suppose there are $k$ internal faces. Then Euler's formula gives $e(H')-(k+1)=\abs{\partial f}-1$. Since each edge is incident with two faces, the sum of face degrees coincides with $2e(H')$. Hence $3k+\abs{\partial^\circ f}\leq 2e(H')$, so $k\leq 2\abs{\partial f}-\abs{\partial^\circ f}$.
Since the length of the list is the number of faces incident with the new vertex, which is at most $k$,
the result follows.\end{proof}
We are now ready to bound $\abs M$ in terms of $\abs B$ and $\abs{B^\circ}$.
\begin{lem}\label{unzip}For each pair $(M,B)\in\pairs_n$ we have $\abs{M}\leq 2n-\abs{B^\circ}$.\end{lem}
\begin{proof}Fix such a pair, and let $T=T(M,B)$ be the plane triangulation fixed above. Recall that $T$ is formed in such a way that 
$B$ is internally connected, and every vertex of $M$ is adjacent to $B$.
Let $C_o$ be the component containing $o$ of $T\setminus B$.  
Removing the vertices of $C_o$ would leave a face $f$ with boundary vertices $B$.

We define an ``unzipping'' operation on $B$ as follows. We imagine that each edge of $B$ has positive width so that each such edge has two edge-sides, where each of them is incident with exactly one face. Moreover, each edge-side has two ends reaching the endvertices of the corresponding edge. Proceed clockwise around the boundary of $f$ along the edge-sides incident with $f$, recording the ends of edges-sides of $f$ which are crossed in a cyclic ordering. Group these edge-ends by the vertex in $B$ which they reach; since $f$ is a face of $T\setminus C_o$, no edge of $T$ inside $f$ connects two vertices of $B$. 

Note that, since $T$ is a triangulation, 
every time a vertex of $B$ is encountered when proceeding around $\partial f$ clockwise, at least one edge-end of $T$ incident with that vertex is crossed. 
Thus the number of groups in the cyclic ordering of edge-ends is precisely the number of entries in the list constructed in \lemma{boundary};
since $\abs B=n$, this list has at most $2n-\abs{B^\circ}$ entries.

We now ``unzip'' $B$ by replacing vertices in $B$ by the entries of the list, so that each vertex which appears more than once
in the list is split into multiple vertices distinguished by list position. We also replace edges in $\partial f$ by edges between 
consecutive entries in the list; this means that any edges of $\partial f$ which were surrounded by $f$ will also be split into two.

There is a one-to-one correspondence between groups of edge-ends of $T$ inside $f$ and entries in the list; we use this correspondence
to replace every edge between $C_o$ and $B$ by an edge between $C_o$ and a specific list entry. This will ensure that the
graph obtained is still plane, and every vertex of $T$ inside $f$ retains its original degree. Finally, remove all
vertices and edges which lie completely outside $f$.
\begin{figure}
\begin{center}
\includegraphics[width=.35\textwidth]{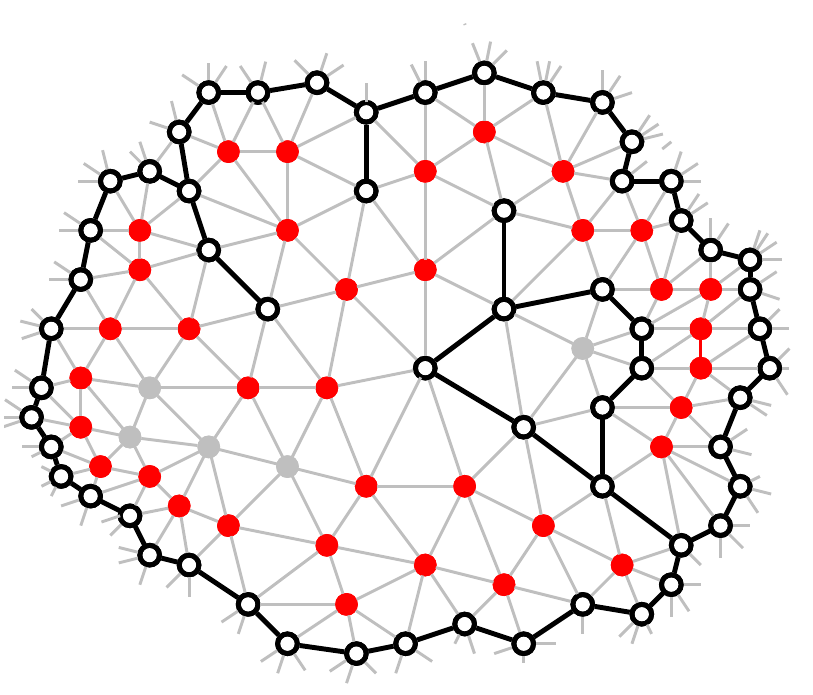}\hspace{.10\textwidth}\includegraphics[width=.35\textwidth]{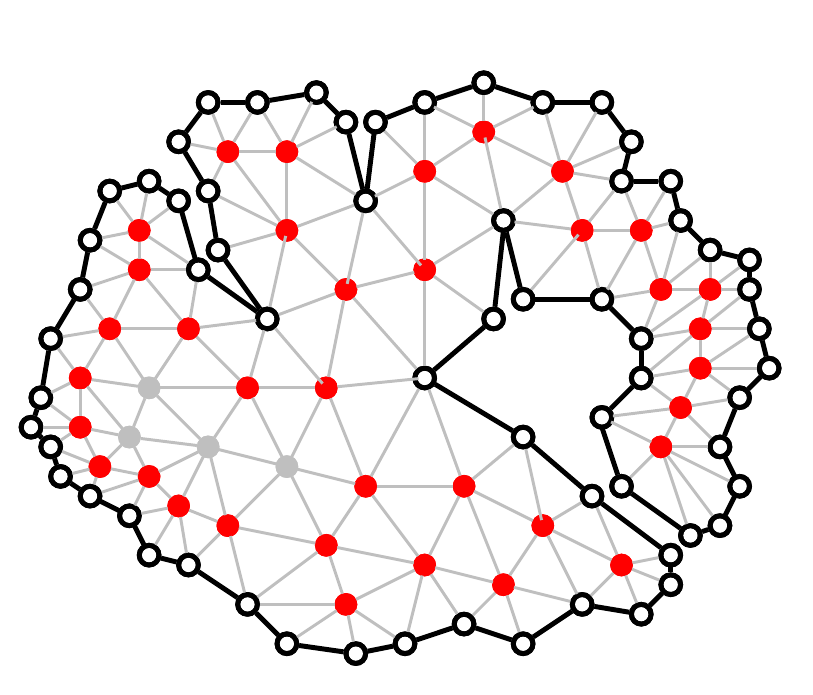}
\end{center}
\caption{Unzipping an outer boundary $B$ (bold vertices and edges); $M$ is shown in red.}\label{fig}
\end{figure}
Figure~\ref{fig} illustrates this unzipping operation.

This produces a disc triangulation with all internal vertices having degree $6$. 
The simple boundary has at most $2n-\abs{B^\circ}$ vertices, and $M$ is precisely the boundary of the internal vertices. 
The required bound now follows from \lemma{iso-layer}.\end{proof}

We now have all the ingredients required for \theorem{first}.
\begin{proof}[Proof of \theorem{first}]
We apply \theorem{main}; by \lemma{unzip} we have \eqref{ratio-condition} with $\alpha=2$, and by \corollary{volume} 
we have \eqref{poly-condition} with $f(n)=\bound$.
\end{proof}

If we further assume that $G$ is non-amenable, we obtain better bounds. 
We say $G$ is \textit{non-amenable} if $\dot h(G)>0$. 
\begin{thm}Let $G$ be a non-amenable plane graph of minimum degree at least $6$, and set $\beta=\dot h(G)>0$. Then $\pc(G)\leq\frac{2+\beta}{3+3\beta}$.\end{thm}
\begin{rem}This result is only interesting for $0<\beta<1$, since for $\beta\geq 1$ it is weaker than the bound of $\frac{1}{1+\beta}$,
due to Benjamini and Schramm \cite{beyond}, which applies to general graphs. 
This is as we would expect, since for $\beta\geq 1$ a stronger form of \lemma{iso-layer} may be deduced directly from the Cheeger constant.\end{rem}
\begin{proof}
Consider some $(M,B)\in\pairs_n$ and the corresponding triangulation $T$. Let $B^\circ\subseteq B$ be the minimum vertex cut separating $o$ from infinity. Note that $B^\circ$ is the boundary of some set $A$ consisting of those vertices of $T$ separated from infinity by $B^\circ$. In particular, $A\supseteq M\cup(B\setminus B^\circ)$. 
Consequently $\abs A\geq\abs M+n-\abs{B^\circ}$, giving
\[\abs{B^\circ}\geq\beta(\abs M+n-\abs{B^\circ}),\]
\ie
\begin{equation}\abs{B^\circ}\geq\frac{\beta}{1+\beta}(\abs M+n)\,.\label{smallalpha}\end{equation}

Unzipping $B$ as in the proof of \lemma{unzip} gives a disc triangulation with boundary of length at most $2n-\abs{B^\circ}$ by \lemma{boundary}, 
and so applying \lemma{iso-layer} and \eqref{smallalpha} gives $\abs M\leq 2n-\frac{\beta}{1+\beta}(\abs{M}+n)$, 
\ie $\abs M\leq\bigbr{\frac{2+\beta}{1+2\beta}}n$. Applying \theorem{main} with $\alpha=\frac{2+\beta}{1+2\beta}$ gives the required result.
\end{proof}

We remark that Angel, Benjamini and Horesh \cite{ABH} proved that for any $r>0$, if a plane triangulation $T$ has minimum degree at least $6$ and every ball of radius $r$ contains a vertex of degree greater than $6$, then $T$ is non-amenable. 

\section{Hyperbolic graphs}\label{hyper}
In this section we consider more stringent degree conditions, motivated by regular hyperbolic graphs. Benjamini and Schramm \cite{beyond} conjecture
that a plane graph $G$ of minimum degree at least $7$ has $\pc<1/2$, and furthermore that there are infinitely many infinite open clusters for every $p\in (\pc,1-\pc)$. We show the first half of this conjecture. In the process we give tight bounds on the vertex Cheeger constant.

Let $G$ be a plane graph with minimum degree at least $d\geq7$, fixed throughout this section.
\begin{lem}\label{hyper-layer}Fix a minimal vertex cut $X$ separating $o$ from infinity. Let $Y$ be the vertices of the finite component of $G\setminus X$ which are adjacent to $X$, and $Z$ be the other vertices of this component. Then $\abs X\geq(d-5)\abs Y+(d-6)\abs Z+5$.\end{lem}
\begin{proof}Consider the induced graph on $X\cup Y\cup Z$. We may add edges if necessary so that this is a disc triangulation with boundary $X$, and we may do this in such a way that the set of vertices in $G\setminus X$ which are adjacent to $X$ is precisely $Y$. Notice that the vertices of $X$ induce a cycle, and the graph induced by $Y\cup Z$ is a connected plane triangulation because $X$ is a minimal vertex cut.

Now add a new vertex $u$ in the unbounded face of this graph adjacent to all vertices in $X$ to obtain a new graph $H$. We now have a triangulation of a sphere with $\abs X+\abs Y+\abs Z+1$ vertices, \ie all faces, including the unbounded one, have degree $3$.
Since each edge is incident with two faces, we can easily deduce that $3F=2E$, where $F$ and $E$ denote the number of faces and edges of $H$. Using Euler's formula we now obtain that $E=3(\abs X+\abs Y+\abs Z)-3$.

Look at faces between $X$ and $Y$ in $H$, \ie faces having two vertices from $X$ and one from $Y$ or vice versa. There are $\abs X$ of the first type (one for each edge between vertices of $X$), because the vertices of $X$ span a cycle, and at least $\abs Y-1$ of the second, because there is at least one face for each edge between vertices of $Y$, and $Y$ spans a connected graph (in fact we only need the ${}-1$ in the case $\abs Y=1$). When walking along $X$ the faces between $X$ and $Y$ appear in a cyclic order because $X$ spans a cycle, hence the number of edges between $X$ and $Y$ coincides with the number of faces between $X$ and $Y$. It follows that there are at least $\abs X+\abs Y-1$ edges between $X$ and $Y$.

Thus we can calculate the sum of the degrees $d(v)$ of $H$ to be 
\[\sum_{v\in V(H)} d(v)\geq\sum_{v\in Y\cup Z}d(v)+\abs X+\abs Y-1+4\abs X\geq5\abs X+(d+1)\abs Y+d\abs Z-1\,.\]
To see that the first inequality holds notice that the sum $\sum_{v\in X\cup\{u\}}d(v)$ counts the edges between $X$ and $Y$ only once and the remaining edges with at least one endvertex in $X$ exactly twice. Now the number of edges connecting vertices of $X$ and the number of edges between $X$ and $u$ are both equal to $\abs{X}$. In the second inequality we used that $d(v)\geq d$.
Since $\sum_{v\in V(H)} d(v)$ is equal to $2E=6(\abs X+\abs Y+\abs Z)-6$ by the handshake lemma, the result follows.
\end{proof}
In particular, we have $(d-5)\abs Y\leq \abs X-5$ and $(d-6)\abs{Y\cup Z}\leq\abs X-5$. In fact, the second inequality can be improved.
\begin{lem}\label{hyper-iso}For $X,Y,Z$ as in \lemma{hyper-layer}, we have $\alpha_d\abs{Y\cup Z}<\abs X-5$, where 
\[\alpha_d=\frac{d - 6 + \sqrt{(d-2)(d-6)}}2\,.\]
In particular, $\alpha_7\approx 1.618$ is the golden ratio.\end{lem}
\begin{proof}We prove this by induction on $\abs Z$. If $Z=\varnothing$ then the required inequality holds since $\alpha_d<d-5$. 
Otherwise, $Y$ contains $k\geq 1$ minimum vertex cuts which separate clusters of vertices in $Z$ from infinity. Write $(Y_i)_{i=1}^k$ for the cuts 
(which may overlap) and $(Z_i)_{i=1}^k$ for the clusters; by \lemma{hyper-layer} $\abs{Y_i}\geq d$ for each $i$.

We may add edges, where necessary, between vertices of $Y$ such that each of the cuts forms a cycle. The auxiliary graph $H$ consisting only of these cycles is outerplanar, \ie all vertices of $H$ belong to the unbounded face of $H$, by definition of $Y$. Notice that $H$ has $k$ faces bounded by the vertex cuts, one unbounded face, and possibly a few more faces that are formed when some vertex cuts overlap in a cyclic way. Let us denote by $\ell$ the number of the latter faces. Moreover, $H$ has $\abs Y$ vertices and its unbounded face has degree at least $\abs Y$, because $H$ is a connected graph and has at least one cycle. Since each edge of $H$ is incident with two faces, the sum of face degrees coincides with $2e(H)$. It follows that $\sum_{i=1}^k\abs{Y_i}\leq 2e(H)-\abs Y-3\ell$, and by
Euler's formula the right hand side of the inequality is equal to $\abs Y+2k-\ell-2$, implying that $\sum_{i=1}^k(\abs{Y_i}-2)\leq\abs Y-2$.

By the induction hypothesis, we have $\alpha_d\abs{Z_i}\leq\abs{Y_i}-5$, and thus 
$\alpha_d\abs Z=\alpha_d\sum_{i=1}^k\abs{Z_i}\leq\sum_{i=1}^k(\abs{Y_i}-5)\leq\abs Y-2-3k\leq\abs Y-5$.

It follows that $\abs Y>\frac{\alpha_d}{1+\alpha_d}(\abs{Y\cup Z})$, so applying \lemma{hyper-layer} we get 
\[\Bigl(d-6+\frac{\alpha_d}{1+\alpha_d}\Bigr)\abs{Y\cup Z}<\abs X-5\,,\]
and since $\alpha_d=d-6+\frac{\alpha_d}{1+\alpha_d}$ the result follows.\end{proof}

\begin{cor}\label{hyper-iso-2}Let $W$ be any finite subset of $V(G)$. Then $\abs{\partial W}\geq\alpha_d\abs{W}$.\end{cor}
\begin{proof}We may assume every vertex in $\partial W$ meets the infinite component of $G\setminus\partial W$, 
since otherwise we may find a larger set with smaller boundary. Now splitting $\partial W$ into minimum vertex cuts surrounding clusters of $W$ as above, 
and applying \lemma{hyper-iso} to each cut, gives the required result.\end{proof}
The following result, which shows that \lemma{hyper-iso} is tight, may be of independent interest.
\begin{thm}\label{cheeger}For each $d\geq 7$ the $d$-regular hyperbolic triangulation $H_{d,3}$ has vertex Cheeger constant $\dot h(H_{d,3})=\alpha_d$.
\end{thm}
\begin{proof}\corollary{hyper-iso-2} immediately gives $\dot h(H_{d,3})\geq\alpha_d$, and so it suffices to exhibit a sequence
of sets $S_n$ satisfying $\abs{\partial S_n}=(\alpha_d+o(1))\abs{S_n}$. In fact the balls $B_n$ have this property. Note that,
for $n\geq 1$, $\partial B_n$ forms a minimum vertex cut, the set of vertices on the unbounded face of $B_n$ is precisely
$B_n\setminus B_{n-1}$, and the induced subgraph on this set is a cycle.

Thus, following the proof of \lemma{hyper-layer}, we obtain $\abs{\partial B_n}=(d-5)\abs{B_n\setminus B_{n-1}}+(d-6)\abs{B_{n-1}}+6$, 
or equivalently, noting that $B_n\cup\partial B_n=B_{n+1}$, 
\begin{equation}\abs{B_{n+1}}-(d-4)\abs{B_n}+\abs{B_{n-1}}=6\,.\label{recurrence}\end{equation}
Standard techniques on recurrence relations imply that the solution of \eqref{recurrence} is given by
$\abs{B_n}=a(1+\alpha_d)^n+b(1+\alpha_d)^{-n}+c$ for suitable constants $a,b,c$ (where clearly $a>0$). In particular, $\abs{\partial B_n}=\abs{B_{n+1}}-\abs{B_n}=(\alpha_d+o(1))\abs{B_n}$, as required.
\end{proof}
\lemma{hyper-iso} therefore shows that among plane graphs with minimum degree $d\geq 7$, $H_{d,3}$ minimises the vertex Cheeger constant. This fact already implies an upper bound on the critical probability, using a result of Benjamini and Schramm \cite{beyond} that $\pc(G)\leq(1+\dot h(G))^{-1}$.
However, combining these facts with our method of interfaces yields better bounds. 

\begin{thm}Let $G$ be a plane graph with minimum degree $d\geq 7$. Then \label{p_c triang}$\pc(G)\leq\frac{2+\alpha_d}{(d-3)(1+\alpha_d)}$.\end{thm}
\begin{rem}For $d=7$, $(1+\alpha_d)^{-1}\approx 0.3820$, whereas $\frac{2+\alpha_d}{(d-3)(1+\alpha_d)}\approx0.3455$.\end{rem}
\begin{proof}Let $M,B$ be an outer interface and its boundary, let $B^\circ$ be the minimal vertex cut part of $B$, and fix a triangulation $T(M,B)$ as described in Section \ref{triang}. We can unzip $B$ as in \lemma{unzip} and then apply \lemma{hyper-layer} to obtain
\begin{equation}(d-5)\abs M\leq2\abs B-\abs{B^\circ}\,.\label{ineq1}\end{equation}
Also, applying \lemma{hyper-iso} to the vertex cut $B^\circ$, we have $\alpha_d(\abs M+\abs{B\setminus B^\circ})\leq\abs{B^\circ}$, \ie
\begin{equation}\alpha_d\abs M\leq(1+\alpha_d)\abs{B^\circ}-\alpha_d\abs B\,.\label{ineq2}\end{equation}
Taking a linear combination of \eqref{ineq1} and \eqref{ineq2} we can cancel $\abs{B^\circ}$ to obtain $(d-5+(d-4)\alpha_d)\abs M\leq(2+\alpha_d)\abs B$. Now \theorem{main} gives $\pc(G)\leq\frac{2+\alpha_d}{(d-3)(1+\alpha_d)}$.
\end{proof}

\section{Hyperbolic quadrangulations}\label{quad}
Let $G$ be a plane graph with no triangular faces, and minimum degree $d\geq 5$, fixed throughout this section. While we will primarily be interested in the case where $G$ is a quadrangulation, 
our results in this section apply more generally to any such graph, even though it is not necessarily possible to create a quadrangulation from such a graph by adding edges. Our first step is an analogue of \lemma{hyper-layer}.
\begin{lem}\label{hyper-quad-layer}Fix a minimal vertex cut $X$ separating $o$ from infinity. Let $Y$ be the vertices of the finite component of $G\setminus X$ which are adjacent to $X$, and $Z$ be the other vertices of this component.
Then $\abs X\geq(d-3)\abs Y+(d-4)\abs Z+3$.\end{lem}
\begin{rem}In fact the proof gives $\abs X\geq(d-3)\abs Y+(d-4)\abs Z+4$ unless $\abs{Y}=1$.\end{rem}
\begin{proof}We may assume $\abs{Y}>1$ since otherwise the result is trivial. Take the induced subgraph on $X\cup Y\cup Z$ and add edges as necessary so that $X$ is a cycle, giving a finite graph $H$. Note that $H$ may have faces of degree $3$. Furthermore, the unbounded face has degree $\abs X$ and, by minimality of $X$, each other face meets $X$ at one vertex, at two vertices with an edge between them, or not at all. Write $A$ for the set of internal faces meeting $X$ along an edge; note that $\abs A=\abs X$. Write $A'$ for the set of faces meeting $X$ at a single vertex. Each face in $A$ has degree at least $3$; let $q$ be the number of faces in $A$ of degree exactly $4$, and $p$ be the number of faces in $A$ of degree at least $5$. 

Once again the sum of face degrees equals $2E$, hence we have $2E\geq\abs X+4(F-\abs X-1)+3\abs A+q+2p=4F-4+q+2p$, where $F$ and $E$ are the number of faces and edges of $H$. By Euler's formula, $F=E-\abs X-\abs Y-\abs Z+2$. Thus $2E\leq 4(\abs X+\abs Y+\abs Z)-q-2p-4$. Also, by the handshake lemma, $2E\geq d\abs{Y}+d\abs{Z}+\sum_{x\in X}d(x)$. The $\abs X$ edges on the unbounded face are double-counted by this sum and there are as many edges between $X$ and $Y$ as faces because $X$ induces a cycle. We can now deduce that $\sum_{x\in X}d(x)=3\abs X+\abs{A'}$. 

We now claim that $\abs{A'}\geq \abs Y-q-p$. Indeed, write $d'(y)$, $y\in Y$ for the number of edges between $y$ and $X$. Since $\abs{A}+\abs{A'}$ coincides with number of edges between $X$ and $Y$, we have $\abs{A}+\abs{A'}=\sum_{y\in Y} d'(y)$. We can rewrite the latter sum as $\abs{Y}+\sum_{y\in Y}(d'(y)-1)$. Notice that $d'(y)>1$ only when $y$ is incident with a triangular face, and moreover $d'(y)-1$ is not less than the number of triangular faces incident with $y$, because $\abs{Y}>1$ and so not all faces incident with $y$ are triangular. Thus $\sum_{y\in Y}(d'(y)-1)\geq \abs{A}-p-q$, which implies that $\abs{A'}\geq \abs{Y}-p-q$.

Consequently we have $3\abs X+(d+1)\abs Y+d\abs Z-q-p\leq 4(\abs X+\abs Y+\abs Z)-q-2p-4$; since $p\geq0$ the result follows.
\end{proof}
We next give an analogue of \lemma{hyper-iso} for this setting; perhaps surprisingly, the same sequence of constants arises.
\begin{lem}\label{hyper-quad-iso}For $X,Y,Z$ as in \lemma{hyper-quad-layer}, we have $\alpha_{d+2}\abs{Y\cup Z}<\abs X-3$.\end{lem}
\begin{proof}We prove this by induction on $\abs Z$. If $Z=\varnothing$ then the required inequality holds since $\alpha_{d+2}<d-3$. 
Otherwise, as in the proof of \lemma{hyper-iso}, $Y$ contains minimum vertex cuts $(Y_i)_{i=1}^k$ separating clusters $(Z_i)_{i=1}^k$,
where $\sum_{i=1}^k\abs{Z_i}=\abs Z$ and $\sum_{i=1}^k(\abs{Y_i}-2)\leq\abs Y-2$.

By the induction hypothesis, we have $\alpha_{d+2}\abs{Z_i}\leq\abs{Y_i}-3$, hence
$\alpha_{d+2}\abs Z=\sum_{i=1}^k \alpha_{d+2}\abs{Z_i}\leq\sum_{i=1}^k(\abs{Y_i}-3)\leq\abs Y-2-k\leq\abs Y-3$.

Consequently $\abs Y>\frac{\alpha_{d+2}}{1+\alpha_{d+2}}(\abs{Y\cup Z})$, and applying \lemma{hyper-quad-layer} gives 
\[\Bigl(d-4+\frac{\alpha_{d+2}}{1+\alpha_{d+2}}\Bigr)\abs{Y\cup Z}<\abs X-3\,,\]
whence the result follows since $\alpha_{d+2}=d-4+\frac{\alpha_d}{1+\alpha_d}$.\end{proof}
Arguing as in the proof of Corollary \ref{hyper-iso-2} we obtain
\begin{cor}\label{hyper-iso-quad}Let $W$ be any finite subset of $V(G)$. Then $\abs{\partial W}\geq\alpha_{d+2}\abs{W}$.\end{cor}

Again, this result is best possible.
\begin{thm}\label{quad-cheeger}For each $d\geq 5$ the $d$-regular hyperbolic quadrangulation $H_{d,4}$ has vertex Cheeger constant $\dot h(H_{d,4})=\alpha_{d+2}$.
\end{thm}
\begin{proof}Again, it suffices to show that $\abs{\partial B_n}=(\alpha_{d+2}+o(1))\abs{B_n}$, where $B_n$ is the ball of radius $n$.

Note that, for $n\geq 1$, $\partial B_n$ forms a minimum vertex cut, the set of vertices in $B_n$ adjacent to $\partial B_n$ is precisely
$B_n\setminus B_{n-1}$, and in the graph obtained by adding a cycle through $\partial B_n$ to the induced subgraph on $B_{n+1}$,
every vertex in $B_n$ has degree $d$, every face meeting $\partial B_n$ along an edge has degree $3$, and every other face has degree $4$.
Thus, following the proof of \lemma{hyper-quad-layer}, we have $q=p=0$ and equality at every step, giving
$\abs{\partial B_n}=(d-3)\abs{B_n\setminus B_{n-1}}+(d-4)\abs{B_{n-1}}+4$, or equivalently
\begin{equation}\abs{B_{n+1}}-(d-2)\abs{B_n}+\abs{B_{n-1}}=4\,.\label{quad-recurrence}\end{equation}
Again, it follows that $\abs{B_n}=a(1+\alpha_{d+2})^n+b(1+\alpha_{d+2})^{-n}+c$ for suitable constants $a,b,c$ with $a>0$, and so $\abs{\partial B_n}=\abs{B_{n+1}}-\abs{B_{n}}=(\alpha_{d+2}+o(1))\abs{B_n}$, as required.
\end{proof}

\lemma{hyper-quad-iso} implies that $\pc(G)\leq(1+\alpha_{d+2})^{-1}$ if $G$ has all vertex degrees at least $d\geq 4$ and all face degrees at least $4$. Our method yields again better bounds.

\begin{thm}Let $G$ be a plane graph with minimum degree $d\geq 5$ and no faces of degree $3$. Then $\pc(G)\leq\frac{(2+\alpha_{d+2})(d-2)}{(d^2-3d+1)(1+\alpha_{d+2})}$.\end{thm}
\begin{rem}For $d=5$, $(1+\alpha_{d+2})^{-1}\approx 0.3820$, whereas $\frac{(2+\alpha_{d+2})(d-2)}{(1+\alpha_{d+2})(d^2-3d+1)}\approx0.3769$.\end{rem}
\begin{proof}
Let $M,B$ be an outer interface and its boundary, and let $B^\circ$ be the minimal vertex cut part of $B$. Delete all vertices of the infinite component of $G\setminus B$. Notice that any vertex of $M$ belongs to a common face with a vertex of $B$. By adding edges if necessary, we can achieve that all faces between $M$ and $B$ have degree $4$ or $5$ and also preserve this property. Then every vertex of $M$ has distance at most $2$ from $B$.

We can argue as in the proof of \lemma{unzip} to unzip $B$ and obtain a new graph $H$ with at most $2\abs B-\abs{B^\circ}$ `boundary' vertices in its unbounded face. To be more precise, we first add a set of edges, which we denote by $S$, to obtain the triangulation $T=T(M,B)$, and then we unzip $B$ as in the proof of \lemma{unzip}. This process gives rise to a correspondence between the edges of the graph obtained after unzipping $B$ and the edges of $T$. The desired graph $H$ is obtained after removing the pre-images of $S$. Write $X$ for the `boundary' vertices of $H$. Let $Y$ be the set of vertices of $H\setminus X$ that are adjacent to $B'$, and $Z$ the remaining vertices of $H$. \lemma{hyper-quad-layer} implies that 
\[2\abs B-\abs{B^\circ} \geq\abs X\geq (d-3)\abs{Y}+(d-4)\abs{M\setminus Y}=\abs Y+(d-4)\abs M.\] 
We now claim that $\abs Y \geq (d-3)(\abs M - \abs Y)+3$.
Indeed, if the graph induced by $Z$ is connected we can apply \lemma{hyper-quad-layer} to $Y\cup Z$, noting that all vertices of $M\setminus Y$ are adjacent to some vertex of $Y$ because $M$ has distance at most $2$ from $B$. If it is not connected we can split $Y$ into minimal vertex cuts $(Y_i)_{i=1}^k$ and $Z$ into its components $(Z_i)_{i=1}^k$, and define $M_i:=(Y_i\cup Z_i)\cap M$. Then we have as above that $\abs{Y_i} \geq (d-3)(\abs{M_i} - \abs{Y_i})+3$. Arguing as in the proof of \lemma{hyper-iso} we obtain that $\sum_{i=1}^k(\abs{Y_i}-2)\leq \abs{Y}-2$. The desired claim follows now from the fact that the sets $M_i\setminus Y_i$ partition $M\setminus Y$.

We can now deduce that $\abs Y \geq \frac{d-3}{d-2} \abs M$, which implies that 
\begin{equation}\Bigl(d-4+\frac{d-3}{d-2}\Bigr)\abs M \leq 2\abs B-\abs{B^\circ}.
\label{ineq3}\end{equation}
Also, applying \lemma{hyper-quad-iso} to the vertex cut $B^\circ$, we have $\alpha_{d+2}(\abs M+\abs{B\setminus B^\circ})\leq\abs{B^\circ}$, \ie
\begin{equation}\alpha_{d+2}\abs M\leq(1+\alpha_{d+2})\abs{B^\circ}-\alpha_{d+2}\abs B\,.\label{ineq4}\end{equation}
Taking a linear combination of \eqref{ineq3} and \eqref{ineq4} we obtain 
\[\Bigl(\alpha_{d+2}+(1+\alpha_{d+2})\Bigl(d-4+\frac{d-3}{d-2}\Bigr)\Bigr)\abs M\leq (2+\alpha_{d+2})\abs{B}.\] Using \theorem{main} we deduce that $\pc(G)\leq\frac{(2+\alpha_{d+2})(d-2)}{(d^2-3d+1)(1+\alpha_{d+2})}$.
\end{proof}

\section*{Acknowledgements}
The authors were supported by the European Research Council (ERC) under
the European Union's Horizon 2020 research and innovation programme (grant 
agreement no.\ 639046), and are grateful to Agelos Georgakopoulos for 
initiating this work.


\begin{thebibliography}{99}
\bibitem{ABH}O. Angel, I. Benjamini and N. Horesh,
An isoperimetric inequality for planar triangulations.
\textit{Discrete Comput.\ Geom.\@} \textbf{59} (2018), 802--809.

\bibitem{cutsets}E. Babson and I. Benjamini,
Cut sets and normed cohomology with applications to percolation. 
\textit{Proc.\ Amer.\ Math.\ Soc.\@} \textbf{127(2)} (1999), 589--597.

\bibitem{Uniformization}I. Benjamini, 
Percolation and coarse conformal uniformization.
In \textit{Unimodularity in randomly generated graphs}, 39--42. 
Amer.\ Math.\ Soc., 2018.

\bibitem{NoPerco} I. Benjamini, R. Lyons, Y. Peres and O. Schramm,
Critical percolation on any nonamenable group has no infinite clusters. 
\textit{Ann.\ Probab.\@} \textbf{27(3)} (1999), 1347--1356.

\bibitem{beyond}I. Benjamini and O. Schramm,
Percolation beyond $\mathbb Z^d$, many questions and a few answers.
\textit{Electron.\ Commun.\ Probab.\@} \textbf{1} (1996), 71--82.

\bibitem{PercoHyperbolic}I. Benjamini and O. Schramm,
Percolation in the hyperbolic plane.
\textit{J. Amer.\ Math.\ Soc.\@} \textbf{14(2)} (2001), 487--507.

\bibitem{PhaseTransitionGroups}H. Duminil-Copin, S. Goswami, A. Raoufi, F. Severo and A. Yadin,
Existence of phase transition for percolation using the Gaussian Free Field. 2018 preprint, arXiv:1806.07733.

\bibitem{analyticity}A. Georgakopoulos and C. Panagiotis,
Analyticity results in Bernoulli percolation.
2018 preprint, arXiv:1811.07404.

\bibitem{ExpGrowth}A. Georgakopoulos and C. Panagiotis,
On the exponential growth rates of lattice animals and interfaces, and new bounds on $p_c$.
2019 preprint, arXiv:1908.03426.

\bibitem{Grimmett}G. Grimmett,
Percolation, Second Edition.
\textit{Grundlehren Math.\ Wiss.\@}, 1999.

\bibitem{HJL}O. H{\"a}ggstr{\"o}m, J. Jonasson and R. Lyons,
Explicit isoperimetric constants and phase transitions in the random-cluster model.
\textit{Ann.\ Probab.\@} \textbf{30(1)} (2002), 443--473.

\bibitem{puMonotonicity}O. H{\"a}ggstr{\"o}m and Y. Peres,
Monotonicity of uniqueness for percolation on Cayley graphs: all infinite clusters are born simultaneously. 
\textit{Probab.\ Theory Related Fields} \textbf{113(2)} (1999), 273--285.

\bibitem{NoPercoExp}T. Hutchcroft,
Critical percolation on any quasi-transitive graph of exponential growth has no infinite clusters. 
\textit{C.\ R.\ Math.\ Acad.\ Sci.\ Paris} \textbf{354(9)} (2016): 944--947.

\bibitem{LyonsBook}R. Lyons and Y. Peres,
Probability on trees and networks.
Cambridge Univ.\ Press, 2016.

\bibitem{Pete}G. Pete,
A note on percolation on $\mathbb Z^d$: Isoperimetric profile via exponential cluster repulsion.
\textit{Electron.\ Commun.\ Probab.\@} \textbf{13} (2008), 377--392.
\end{thebibliography}
\end{document}